\documentclass{amsart}

\usepackage{hyperref}
\usepackage{amsthm,latexsym,amssymb,amsmath}
\usepackage{lscape,array}
\usepackage{times,mathptmx}
\renewcommand{\H}{\mathcal{H}}
\newcommand{\F}{\mathcal{F}}
\newcommand{\D}{\mathcal{D}}

\newcommand{\R}{\mathbb{R}}
\newcommand{\C}{\mathbb{C}}
\newcommand{\N}{\mathbb{N}}
\newcommand{\<}{\left\langle}
\renewcommand{\>}{\right\rangle}
\renewcommand{\phi}{\varphi}
\newcommand{\na}{\nabla}

\newcommand{\eps}{\varepsilon}
\newcommand{\vol}{\mathrm{vol}}

\newcommand{\res}{\mathrm{res}}
\newcommand{\supp}{\operatorname{supp}}
\newcommand{\oo}{\mathrm{o}}
\newcommand{\OO}{\mathrm{O}}

\newtheorem{thm}{Theorem}
\newtheorem{cor}[thm]{Corollary}

\newtheorem{lem}[thm]{Lemma}

\theoremstyle{definition}
\newtheorem{ex}[thm]{Example}
\newtheorem{rem}[thm]{Remark}
\newtheorem{dfn}[thm]{Definition}

\long\def\symbolfootnote[#1]#2{\begingroup%
\def\thefootnote{\fnsymbol{footnote}}\footnote[#1]{#2}\endgroup} 

\title{Some properties of solutions to weakly hypoelliptic equations}
\author{Christian B\"ar}
\address{Universit\"at Potsdam, Institut f\"ur Mathematik, Am Neuen Palais 10, 14469 Potsdam, Germany}
\email{baer@math.uni-potsdam.de}
\urladdr{http://geometrie.math.uni-potsdam.de/}


\begin{document}

\begin{abstract}
A linear different operator $L$ is called weakly hypoelliptic if any local solution $u$ of $Lu=0$ is smooth. 
We allow for systems, that is, the coefficients may be matrices, not necessarily of square size.
This is a huge class of important operators which cover all elliptic, overdetermined elliptic, subelliptic and parabolic equations.

We extend several classical theorems from complex analysis to solutions of any weakly hypoelliptic equation: the Montel theorem providing convergent subsequences, the Vitali theorem ensuring convergence
of a given sequence and Riemann's first removable singularity theorem.
In the case of constant coefficients we show that Liouville's theorem holds, any bounded solution must be constant and any $L^p$-solution must vanish.
\end{abstract}

\maketitle

\section{Introduction}

Hypoelliptic partial differential equations form a huge class of linear PDEs many of which are very important in applications.
This class contains all elliptic, overdetermined elliptic, subelliptic and parabolic equations.
Recall that a linear differential operator $L$ is called hypoelliptic if any solution $u$ to $Lu=f$ is smooth wherever $f$ is smooth.
The study of hypoelliptic operators was initiated by H\"ormander and others, see e.g.\ \cite{H01,H02,M,Tr}.

We generalize this class of operators even further by only demanding that any solution $u$ to $Lu=0$ be smooth.
We call such operators \emph{weakly hypoelliptic}.
This is not to be confused with partially hypoelliptic operators as introduced by G\aa rding and Malgrange \cite{GM} nor with the almost hypoelliptic operators due to Elliott \cite{E}.
We show by example that the class of weakly hypoelliptic operators is strictly larger than that of hypoelliptic operators.
The example of a weakly hypoelliptic but non-hypoelliptic operator that we give is defined on $\R^2$ and is overdetermined elliptic on $\R^2\setminus \{0\}$.
It is of first order and its principal symbol vanishes at $0$.
Thus the class of weakly hypoelliptic operators allows for a certain degeneracy of the principal symbol on ``small sets'' and might be of interest for geometric applications.

Holomorphic functions are the solutions to the Cauchy-Riemann equations which are elliptic in the case of one variable and overdetermined elliptic in the case of several variables.
In any case, they are characterized as solutions to certain hypoelliptic PDEs.
We show that the solutions to any weakly hypoelliptic equation share some of the nice properties of holomorphic functions which are familiar from classical complex analysis.

Montel's theorem says that a locally bounded sequence of holomorphic functions subconverges to a holomorphic function. 
This does not hold for real-analytic functions.
For instance, the sequence $u_j(x)=\cos(jx)$ is a uniformly bounded sequence of real-analytic functions on $\R$, but does not have a convergent subsequence, see e.g.\ \cite[Ex.~1.4.34]{S}.
We show that even a slightly stronger version of the Montel theorem holds for solutions to any weakly hypoelliptic equation:
Any locally $L^1$-bounded sequence subconverges in the $C^\infty$-topology to a solution (Theorem~\ref{thm:montel}).
The Vitali theorem for holomorphic functions has a similar generalization (Theorem~\ref{thm:vitali}).
For hypoelliptic equations this has been known for many decades and has motivated the study of so-called Montel spaces in functional analysis.

In case the underlying domain is $\R^n$ and the weakly hypoelliptic operator has constant coefficients and satisfies a weighted homogeneity condition, we show that the Liouville theorem holds: any bounded solution must be constant (Theorem~\ref{thm:liouville}) and any $L^p$-solution must be zero (Theorem~\ref{thm:Lp-liouville}).
This applies to powers of the Laplace and Dirac operators but also to powers of the heat operator.
In the proof we use a simple scaling argument and apply the general Montel theorem.

Finally, we generalize Riemann's first removable singularity theorem and show that a solution to a weakly hypoelliptic equation can be extended across a submanifold $S$ of sufficiently high codimension provided the solution is locally bounded near $S$ (Theorem~\ref{thm:riemann} and Corollary~\ref{cor:riemann}).

The general setup is such that we consider a linear differential operator $L$ acting on sections of vector bundles. 
So, locally $L$ describes a system of linear PDEs with smooth coefficients.
These coefficients may be matrices of not necessarily square size.
Readers who are not too fond of geometric terminology may simply replace ``manifolds'' by ``open subsets of $\R^n$'' and ``sections of vector bundles'' by ``vector-valued functions''.

The classical proofs of these theorems for holomorphic functions are typically based on special properties of holomorphic functions such as Cauchy's integral formula.
Therefore they may create the misleading impression that these theorems are also very special for holomorphic functions.
In the contrary, the above mentioned theorems remain true for all solutions of the largest class of linear PDEs where one could hope for them to hold.
Moreover, as we will see, the proofs of the general statements are actually rather simple.

\emph{Acknowledgment.}
The author likes to thank Thomas Krainer and Elmar Schrohe for helpful discussions.
Moreover, he thanks Sonderforschungsbereich 647 funded by Deutsche Forschungsgemeinschaft for financial support.

\section{Weakly Hypoelliptic Operators}

Let $M$ be an $n$-dimensional differentiable manifold equipped with a smooth positive $n$-density $\vol$.
Let $E\to M$ and $F\to M$ be real vector bundles.
If the bundles are complex we simply forget the complex structure and consider them as real bundles.
We denote the spaces of smooth sections by $C^\infty(M,E)$ and $C^\infty(M,F)$, respectively.
Let $L: C^\infty(M,E) \to C^\infty(M,F)$ be a linear differential operator of order $k\in\N$.
The fact that smooth sections are mapped to smooth sections encodes the smoothness of the coefficients of $L$ in local coordinates.
The operator $L$ restricts to a linear map $\D(M,E) \to \D(M,F)$ where $\D$ stands for compactly supported smooth sections.

Let $E^*\to M$ and $F^*\to M$ be the dual bundles.
Given $L$, there is a unique linear differential operator $L^*: C^\infty(M,F^*) \to C^\infty(M,E^*)$, the \emph{formally dual operator}, characterized by 
\[
\int_M \<Lu,\phi\>\vol = \int_M \<u,L^*\phi\>\vol
\]
for all $u\in C^\infty(M,E)$ and $\phi\in C^\infty(M,F^*)$ such that $\supp(u) \cap \supp(\phi)$ is compact.
Here $\<\cdot,\cdot\>$ denote the canonical pairing of $E$ and $E^*$ and of $F$ and $F^*$.

We extend $L$ to an operator, again denoted by $L$, mapping distributional sections to distributional sections, $L:\D'(M,E) \to \D'(M,F)$ by
\[
(Lu)[\phi] := u[L^*\phi]
\]
for all $u\in\D'(M,E)$ and $\phi\in\D(M,F^*)$, compare e.g.\ \cite[Sec.~1.1.2]{BGP}.
Here we denote by $u[\psi]$ the evaluation of the distribution $u\in\D'(M,E)$ on the test section $\psi\in\D(M,E^*)$.

The differential operator $L$ is called \emph{hypoelliptic} if for any open subset $\Omega\subset M$ and any $u\in\D'(\Omega,E)$ such that $Lu$ is smooth we have that $u$ is smooth.
Since we will be interested in solutions of $Lu=0$ only, we make the following definition:
The differential operator $L$ is called \emph{weakly hypoelliptic} if for any open subset $\Omega\subset M$ any $u\in\D'(\Omega,E)$ satisfying $Lu=0$ must be smooth.
H\"ormander's work \cite[Ch.~III]{H01} (see also the proof of Theorem~2.1 in \cite[p.~63]{T}) shows that for operators with constant coefficients over $M=\R^n$ hypoellipticity and weak hypoellipticity are equivalent, at least in the scalar case, i.e.\ if the coefficients of $L$ are scalars rather than matrices.
It seems likely that this is also true if the coefficients are constant matrices.
In general however, if the coefficients are variable, the class of weakly hypoelliptic operators is strictly larger than that of hypoelliptic operators.

\begin{ex}\label{ex:weak}
Let $M=\R^2$, let $E$ be the trivial real line bundle and $F$ the trivial $\R^2$-bundle.
The operator $L=(L_1,L_2):C^\infty(\R^2,\R) \to C^\infty(\R^2,\R^2)$ is given by 
\[
L_1 = r\frac{\partial}{\partial r}-2 = x\frac{\partial}{\partial x} +  y\frac{\partial}{\partial y}-2,
\quad
L_2 = r\frac{\partial}{\partial \theta} =  -y\frac{\partial}{\partial x} +  x\frac{\partial}{\partial y}.
\]
Here $(x,y)$ are the usual Cartesian coordinates while $(r,\theta)$ denote polar coordinates, $x=r\cos(\theta)$, $y=r\sin(\theta)$.
On $\R^2\setminus\{0\}$ the operator $L$ is overdetermined elliptic (see below) and hence hypoelliptic.
But $L$ is not hypoelliptic on $\R^2$ because $u=r^2\log r$ is not smooth while $Lu=(r^2,0)$ is smooth.

We check that $L$ is weakly hypoelliptic.
Regularity is an issue at the origin only.
Let $u\in\D'(\Omega,\R)$ with $Lu=0$ where $\Omega$ is an open disk in $\R^2$ centered at the origin.
Then $u$ is smooth on $\Omega\setminus\{0\}$.
From $L_1u=0$ we see that $u=\alpha(\theta)r^2$ on $\Omega\setminus\{0\}$ and $L_2u=0$ shows that $\alpha$ does not depend on $\theta$.
Hence $u=\alpha r^2= \alpha\cdot(x^2+y^2)$ on $\Omega\setminus\{0\}$.
Subtracting this smooth function we may w.l.o.g.\ assume that $\supp(u)\subset\{0\}$.
In this case, $u$ is a linear combination of the delta function and its derivatives, 
\[
u=\sum_{i,j}\beta_{ij}\frac{\partial^{i+j}\delta_0}{\partial x^i\partial y^j}.
\]
Fix $i_0$ and $j_0$ and choose a test function $\varphi\in C^\infty_c(\Omega,\R)$ which coincides with the monomial $x^{i_0}y^{j_0}$ on a neighborhood of the origin.
Then we see
\begin{align*}
0
&= 
L_1u[\varphi] \\
&=
u\left[-\frac{\partial (x\varphi)}{\partial x}-\frac{\partial (y\varphi)}{\partial y}-2\varphi\right]\\ 
&=
u[-(i_0+1+j_0+1+2)\cdot\varphi]\\
&= -(4+i_0+j_0)\cdot\sum_{i,j}\beta_{ij} \cdot(-1)^{i+j}\cdot\delta_0\left[ \frac{\partial^{i+j}\varphi}{\partial x^i\partial y^j}\right] \\
&= 
-(4+i_0+j_0)\cdot\beta_{i_0j_0}\cdot (-1)^{i_0+j_0}\cdot i_0!\cdot j_0!.
\end{align*}
Thus $\beta_{i_0j_0}=0$ for all $i_0$ and $j_0$ and therefore $u=0$.
To summarize, we have seen that $L$ is weakly hypoelliptic, but not hypoelliptic.
\end{ex}

For any weakly hypoelliptic operator we denote the kernel of $L:\D'(M,E)\to\D'(M,F)$ by $\H(M,L)\subset C^\infty(M,E)$.

For $j\in\N_0$ and any relatively compact measurable subset $A\subset M$ we define the $C^j$-norm of $u\in C^j(M,E)$ by
\[
\|u\|_{C^j(A)} := \sup_{x\in A}\,\max\left(|\na^ju(x)|,\ldots,|\na u(x)|,|u(x)|\right) .
\]
Here we have tacitly equipped $E$ and the tangent bundle $TM$ with a Riemannian metric and a connection $\na$.
These data induce fiberwise norms and connections on the bundles $T^*M\otimes \cdots \otimes T^*M \otimes E$.
Note that $\na^ju$ is a section of $\underbrace{T^*M\otimes \cdots\otimes  T^*M}_{j\,\,\mathrm{factors}} \otimes E$.
Since $A$ is relatively compact, different choices of metric and connections yield equivalent $C^j$-norms.

If $K\subset M$ is a compact subset, then we denote by $C^j(K,E)$ the set of all restrictions to $K$ of $j$-times continuously differentiable sections, defined on an open neighborhood of $K$.
Equipped with the norm $\|\cdot\|_{C^j(K)}$, $C^j(K,E)$ becomes a Banach space.

Similarly, we define the $L^p$-norm for $1 \leq p < \infty$ by
\[
\|u\|_{L^p(A)}^p := \int_A |u|^p \vol(x) .
\]
Again, different choices of the metric and the volume density yield equivalent $L^p$-norms.
The $L^\infty$-norm extends the $C^0$-norm to the space of all essentially bounded measurable sections.
The starting point are the following hypoelliptic estimates (compare \cite[p.~331, Prop.~2]{M1} for the hypoelliptic case):

\begin{lem}[Hypoelliptic estimates]\label{lem:EllEst}
Let $L$ be weakly hypoelliptic.
Then for any $j\in\N$, for any compact subset $K\subset M$ and any open subset $\Omega\subset M$ containing $K$ there is a constant $C>0$ such that
\begin{equation}
\|u\|_{C^j(K)} \leq C\cdot \|u\|_{L^1(\Omega)}
\label{eq:EllEst}
\end{equation}
for all $u\in\H(M,L)$.
\end{lem}

\begin{proof}
Let $V$ be the kernel of the continuous linear map $L:L^1(\Omega,E) \subset \D'(\Omega,E) \to \D'(\Omega,F)$.
Hence $V$ is a closed subspace of $L^1(\Omega,E)$ and thus a Banach space with the norm $\|\cdot\|_{L^1(\Omega)}$.
Since $L$ is weakly hypoelliptic we have $V\subset C^\infty(\Omega,E)$.
Thus we get the linear restriction map $\res:V\to C^j(K,E)$, $u\mapsto u|_K$.

This map is closed.
Namely, let $u_i\to u$ with respect to $\|\cdot\|_{L^1(\Omega)}$ and $\res(u_i) \to v$ with respect to $\|\cdot\|_{C^j(K)}$.
Then we also have $\res(u_i) \to v$ with respect to $\|\cdot\|_{L^1(K)}$ and therefore $\res(u)=v$ which proves closedness.

The closed graph theorem implies that $\res$ is bounded which is \eqref{eq:EllEst} for all $u\in\H(\Omega,L)$.
In particular, \eqref{eq:EllEst} holds for all  $u\in\H(M,L)$.
\end{proof}

\begin{cor}
Let $L$ be a weakly hypoelliptic operator over a compact manifold $M$ (without boundary).
Then $\H(M,L)$ is finite-dimensional.
\end{cor}

\begin{proof}
Since $M$ is compact we can take $K=\Omega=M$ in Lemma~\ref{lem:EllEst}.
Thus the $C^0$-norm and the $C^1$-norm are equivalent on $\H(M,L)$.
By the Arzel\`a-Ascoli theorem the embedding $C^1(M,E)\hookrightarrow C^0(M,E)$ is compact.
Hence the identity map on $\H(M,L)$ is compact, thus $\H(M,L)$ is finite-dimensional.
\end{proof}

A differential operator $L$ is called \emph{elliptic} if the principal symbol $\sigma_L(\xi)$ is invertible for all nonzero covectors $\xi\in T^*M$.
Elliptic regularity theory implies that all elliptic operators are hypoelliptic, see \cite[Thm.~11.1.10]{H2} or \cite[Ch.~III, {\S}5]{LM}.
The class of elliptic operators contains many examples of high importance for applications such as the Laplace and Dirac operator.

More generally, if the principal symbol $\sigma_L(\xi)$ is injective instead of bijective for all nonzero covectors $\xi\in T^*M$, then one calls $L$ \emph{overdetermined elliptic}.
In this case $L^*L$ is elliptic where $L^*$ denotes the formal adjoint of $L$.
Now if $Lu|_\Omega$ is smooth, so is $L^*Lu|_\Omega$ and hence $u|_\Omega$ is smooth by elliptic regularity.
Therefore overdetermined elliptic operators are hypoelliptic as well.

Another way to generalize elliptic operators within the class of hypoelliptic operators is to consider \emph{subelliptic operators with a loss of $\delta$ derivatives} where $\delta\in(0,1)$.
These operators can also be characterized by a condition on their principal symbol, see \cite[Ch.~XXVII]{H4} for details.

If $L$ is \emph{parabolic}, e.g.\ if $L$ describes the heat equation on a Riemannian manifold, then parabolic regularity using anisotropic Sobolev spaces shows that $L$ is hypoelliptic, see e.g.\ \cite[Sec.~6.4]{G}.

In contrast, hyperbolic differential operators, e.g.\ those which describe wave equations, are not hypoelliptic.

Table~1 is a (very incomplete) table of examples for hypoelliptic operators relevant for applications:

\begin{landscape}
\renewcommand{\arraystretch}{1.3}
\begin{tabular}{|>{\centering}m{30mm}|>{\centering}m{28mm}|>{\centering}m{29mm}|>{\centering}m{29mm}|>{\centering}m{30mm}|>{\centering}m{14mm}|c|}
\hline
$L$ & $M$ & $E$ & $F$ & $\H(M,L)$ & type & reference\\
\hline\hline
$\partial/\partial\bar z$, Cauchy-Riemann & open subset of $\C$ & trivial $\C$-line bundle & trivial $\C$-line bundle & holomorphic functions & elliptic & \cite[Sec.~I.5]{FB} \\
\hline
$\bar\partial_E$ & complex manifold & holomorphic vector bundle & $E\otimes T^*M^{0,1}$ & holomorphic sections& overdet.\ elliptic & \cite[Sec.~0.5]{GH}\\
\hline
$\na^*\na+l.o.t.$, Laplace-type & Riemannian manifold & Riem.\ vector bundle with connection & $E$ & harmonic sections & elliptic & \cite[Sec.~2.1]{BGV} \\
\hline
$\Delta^2$, bi-Laplacian & Riemannian manifold & trivial $\R$-line bundle& trivial $\R$-line bundle & biharmonic functions & elliptic & \\
\hline
$D$, Dirac & Riemannian spin manifold & spinor bundle & spinor bundle & harmonic spinors & elliptic & \cite[Ch.~II, {\S}5]{LM} \\
\hline
$\partial/\partial t - \Delta$,\hspace{10mm} heat operator & interval $\times$ Riemannian manifold & trivial $\R$-line bundle& trivial $\R$-line bundle &  & parabolic & \cite[Sec.~6.4]{G}\\
\hline
$\Delta_X$, sub-Laplacian & sub-Riemannian manifold (bracket generating) &  trivial $\R$-line bundle& trivial $\R$-line bundle &  & subelliptic & \cite[Thm.~1.1]{H03}\\
\hline
$\sum_j X_j^*X_j$, where $X_j=$ complex vector fields (bracket cond.) & $\R^n$ &  trivial $\C$-line bundle& trivial $\C$-line bundle &  & Sobolev-subell. & \cite[Thm.~C]{K}\\
\hline
Bismut's hypo- elliptic Laplacian & $T^*X$ &  $\bigoplus_{j=0}^n\bigwedge^jT^*T^*X$ & $E$ &  & hypo\-elliptic & \cite[Ch.~3]{B}\\
\hline
$\partial^2/\partial x^2 + x\partial/\partial y - \partial/\partial t$, Kolmogorov & $(0,\infty)\times\R^2$ & trivial $\R$-line bundle& trivial $\R$-line bundle &  & hypo\-elliptic & \cite[Sec.~22.2]{H3}\\ 
\hline
\end{tabular}
\begin{center}
\textbf{Tab.~1} Examples of hypoelliptic differential operators
\end{center}
\end{landscape}

\section{Convergence Results}

Let $1\leq p \leq \infty$.
A family $\F\subset C^\infty(M,E)$ is called \emph{locally $L^p$-bounded} if for each compact subset $K\subset M$ we have
\[
\sup_{u\in\F}\|u(x)\|_{L^p(K)} < \infty.
\]
Let $1\leq p \leq q < \infty$.
By H\"older's inequality $\|u\|_{L^p(K)} \leq \vol(K)^{\frac{q-p}{pq}}\cdot \|u\|_{L^q(K)}$.
For $q=\infty$ we have $\|u\|_{L^p(K)} \leq \vol(K)^{\frac{1}{p}}\cdot \|u\|_{L^\infty(K)}$.
Therefore local $L^q$-boundedness implies local $L^p$-boundedness whenever $q\geq p$.
In particular, local $L^1$-boundedness is the weakest of these boundedness conditions.

We say that a sequence $(u_j)$ in $C^\infty(M,E)$ converges in the $C^\infty$-topology if the restriction to any compact subset $K\subset M$ converges in every $C^j$-norm.
In other words, the sections and all their derivatives converge locally uniformly.

\begin{thm}[Generalized Montel Theorem]\label{thm:montel}
Let $L$ be a weakly hypoelliptic operator.
Then any locally $L^1$-bounded sequence $u_1,u_2,\ldots\in\H(M,L)$ has a subsequence which converges in the $C^\infty$-topology to some $u\in\H(M,L)$.
\end{thm}

The experts will notice that this is a direct consequence of Lemma~\ref{lem:EllEst} because we have the following

\begin{proof}[Short proof.]
By Lemma~\ref{lem:EllEst} the sequence $u_1,u_2,\ldots$ is bounded with respect to the $C^\infty$-topology (in the sense of topological vector spaces).
Since $C^\infty(M,E)$ is known to be a Montel space \cite[p.~148, Cor.~2]{T2} so is the closed subspace $\H(M,L)$ (when equipped with the $C^\infty$-topology).
This means that bounded closed subsets are compact, hence $u_1,u_2,\ldots$ has a convergent subsequence.
\end{proof}

For those unfamiliar with the theory of Montel spaces we can also provide the following

\begin{proof}[Elementary proof.]
Let $K\subset M$ be a compact subset.
We choose an open, relatively compact subset $\Omega\Subset M$ containing $K$.
We fix $j\in\N$.
Since, by Lemma~\ref{lem:EllEst},
\[
\|u_\nu\|_{C^{j+1}(K)} \leq C\cdot \|u_\nu\|_{L^1(\Omega)} \leq C\cdot \|u_\nu\|_{L^1(\overline\Omega)} \leq C'
\]
the sequence $(u_\nu)_\nu$ is bounded in the $C^{j+1}$-norm.
By the Arzel\`a-Ascoli theorem there is a subsequence which converges in the $C^j$-norm over $K$.
The diagonal argument yields a subsequence which converges in all $C^j$-norms over $K$, $j\in\N$.

Now we exhaust $M$ by compact sets $K_1 \subset K_2 \subset K_3 \subset \cdots \subset M$.
We have seen that over each $K_\mu$ we can pass to a subsequence converging in all $C^j$-norms.
Applying the diagonal argument once more, we find a subsequence converging in all $C^j$-norms over all $K_\mu$.
Thus we found a subsequence which converges in the $C^\infty$-topology to some $u\in C^\infty(M,E)$.

Since $L: C^\infty(M,E) \to C^\infty(M,F)$ is (sequentially) continuous with respect to the $C^\infty$-topology, we have $Lu=0$, i.e.\ $u\in\H(M,L)$.
\end{proof}

The generalized Montel theorem applies to all examples listed in Table~1.
Even in the case of holomorphic functions (in one or several variables) Theorem~\ref{thm:montel} is a slight improvement over  the classical Montel theorem because the classical condition of local $L^\infty$-boundedness is replaced by the weaker condition of local $L^1$-boundedness.
The standard proof of the classical Montel theorem uses the Cauchy integral formula to show equicontinuity and then applies the Arzel\`a-Ascoli theorem, see e.g.\ \cite[Sec.~1.4.3]{S}.

In the case of harmonic functions on a Euclidean domain the Montel theorem is also classical.
One can use estimates based on the Poisson kernel to show equicontinuity and then apply the Arzel\`a-Ascoli theorem \cite[p.~35, Thm.~2.6]{ABR}.

The Montel theorem provides a criterion for the existence of a convergent subsequence.
The next theorem provides a sufficient criterion which ensures that a given sequence converges itself.

\begin{dfn}
Let $L$ be a weakly hypoelliptic operator on $M$.
A subset $A\subset M$ is called a \emph{set of uniqueness} for $L$ if for any $u\in\H(M,L)$ the condition $u|_A=0$ implies $u=0$.
\end{dfn}

Every dense subset $A$ of $M$ is a set of uniqueness because $\H(M,L)\subset C^0(M,E)$.

For holomorphic function of one variable, i.e.\ $L=\frac{\partial}{\partial \bar z}$, the set $A$ is a set of uniqueness if it has an accumulation point in $M\subset \C$.

Many (but not all) important elliptic operators have the so-called \emph{weak unique continuation property}.
This means that if $A$ has nonempty interior, then it is a set of uniqueness provided $M$ is connected.
Laplace- and Dirac-type operators belong to this class.

\begin{thm}[Generalized Vitali Theorem]\label{thm:vitali}
Let $L$ be a weakly hypoelliptic operator and let $A\subset M$ be a set of uniqueness for $L$.
Let  $u_1,u_2,\ldots\in\H(M,L)$ be a locally $L^1$-bounded sequence.
Suppose that the pointwise limit $lim_{j\to\infty}u_j(x)$ exists for all $x\in A$.

Then $(u_j)_j$ converges in the $C^\infty$-topology to some $u\in\H(M,L)$.
\end{thm}

\begin{proof}
By Theorem~\ref{thm:montel}, every subsequence of $(u_j)_j$ has a subsequence for which the assertion holds.
The limit functions for these subsequences are in $\H(M,L)$ and coincide on $A$, hence they all agree.
Hence $(u_j)_j$ has a unique accumulation point $u\in\H(M,L)$.

If the sequence $(u_j)_j$ itself did not converge to $u$ then we could extract a subsequence staying outside a $C^\infty$-neighborhood of $u$.
But this subsequence would again have a subsequence converging to $u$, a contradiction.
\end{proof}

\section{Liouville Property}

We now concentrate on the case $M=\R^n$.
All vector bundles over $\R^n$ are trivial, so sections can be identified with functions $\R^n\to\R^N$.
Hence the coefficients of the differential operator $L$ are $N\times N'$-matrices.
If these matrices do not depend on the point $x\in\R^n$ we say that $L$ has constant coefficients.
The Laplace operator $\Delta$, the Dirac operator $D$ and the heat operator $\frac{\partial}{\partial x_1}-\sum_{j=2}^n\frac{\partial^2}{\partial x_j^2}$ are examples of hypoelliptic operators on $\R^n$ with constant coefficients.

Let $P$ be a polynomial in $n$ real variables.
Here $P$ is allowed to have matrix-valued coefficients of fixed size.
Let $w=(w_1,\ldots,w_n)$ with $w_j>0$.
We call $P$ \emph{weighted homogeneous} with weight $w$ if $P(t^{w_1}x_1,\ldots,t^{w_n}x_n)= t^kP(x_1,\ldots,x_n)$ for some $k$ and all $t\in\R$, $x=(x_1,\ldots,x_n)\in\R^n$.
The corresponding differential operator with constant coefficients $L=P(\frac{\partial}{\partial x})=P(\frac{\partial}{\partial x_1},\ldots,\frac{\partial}{\partial x_n})$ is then also called weighted homogeneous.
The Dirac and Laplace operator are examples for weighted homogeneous differential operators with weight $w=(1,\ldots,1)$ as well as the heat operator (weight $w=(2,1,\ldots,1)$).

We can now state the following Liouville type theorem.

\begin{thm}[Generalized Liouville Theorem]\label{thm:liouville}
Let $L$ be a weakly hypoelliptic operator over $\R^n$.
Suppose that $L$ has constant coefficients in $N'\times N$-matrices and is weighted homogeneous.

Then each bounded function in $\H(\R^n,L)$ must be constant.
\end{thm}

\begin{proof}
Let $u\in\H(\R^n,L)$ be bounded.
For $\eps>0$ we put $u_\eps(x):=u(\eps^{-w_1} x_1,\ldots,\eps^{-w_n} x_n)$.
Since $u$ is bounded, the family $(u_\eps)$ is uniformly bounded.
Moreover, $(Lu_\eps)(x)=\eps^{-k} Lu(\eps^{-w_1} x_1,\ldots,\eps^{-w_n} x_n)=0$ so that $u_\eps\in\H(\R^n,L)$.
By Theorem~\ref{thm:montel}, there is a sequence $\eps_j\searrow 0$ such that $u_{\eps_j}$ converges locally uniformly to some $v\in\H(\R^n,L)$.
We observe $u_\eps(0)=u(0)$ and hence $v(0)=u(0)$.

Fix $x\in\R^n$.
For $\eps>0$ we put $x_\eps := (\eps^{w_1} x_1,\ldots,\eps^{w_n} x_n)$.
Then $u_\eps(x_\eps)=u(x)$ and $x_\eps \to 0$ as $\eps \searrow 0$.
Locally uniform convergence yields $u(x)=u_{\eps_j}(x_{\eps_j}) \to v(0) = u(0)$, hence $u(x)=u(0)$, so $u$ is constant.
\end{proof}

\begin{ex}
We directly recover the classical Liouville theorems for holomorphic and for harmonic functions.
In the case of bounded harmonic functions Nelson gave a particularly short proof based on the mean value property \cite{N}.
In fact, for harmonic functions it suffices to assume that they are bounded from below (or from above) \cite[Thm.~3.1]{ABR}.
This cannot be deduced from Theorem~\ref{thm:liouville} but the theorem also applies to biharmonic functions on $\R^n$ or to solutions of higher powers of $\Delta$.
The function $u(x)=|x|^2$ is biharmonic, bounded from below and nonconstant.
Hence unlike for harmonic functions we need to assume boundedness from above and from below to conclude that a biharmonic function is constant.

Similarly, bounded harmonic spinors on $\R^n$ must be constant.
\end{ex}

\begin{rem}
Here is a silly argument why all bounded polynomials on $\R^n$ must be constant.
Given such a polynomial $u$ choose $\ell\in\N$ larger than half the degree of $u$.
Then $\Delta^\ell u=0$ and Theorem~\ref{thm:liouville} applies.
\end{rem}

\begin{ex}
Theorem~\ref{thm:liouville} also applies to the heat operator.
Bounded solutions to the heat equation on $\R^n=\R\times\R^{n-1}$ must be constant.
Note that there do exist nontrivial solutions on $\R^n$ which vanish for $x_1\leq 0$ \cite[pp.~211-213]{J}.
They are unbounded on $\R^{n-1}$ for each $x_1>0$ however.

Moreover, Theorem~\ref{thm:liouville} applies to powers of the heat operator.
So, for instance, bounded solutions to 
\[
\left(\frac{\partial}{\partial x_1}-\sum_{j=2}^n\frac{\partial^2}{\partial x_j^2}\right)^2 u =0
\] 
must be constant.
\end{ex}

\begin{rem}
Theorem~\ref{thm:liouville} does not hold for hyperbolic operators.
The function $u(x_1,x_2)=\sin(x_1)\sin(x_2)$ is non-constant, bounded and solves the wave equation $\frac{\partial^2 u}{\partial x_1^2} - \frac{\partial^2 u}{\partial x_2^2}=0$.
Thus Theorem~\ref{thm:liouville} does not extend to partially hypoelliptic operators in the sense of G\aa rding and Malgrange \cite{GM}.
\end{rem}

\begin{thm}[Generalized Liouville Theorem, $L^p$-version]\label{thm:Lp-liouville}
Let $1\leq p < \infty$.
Let $L$ be a weakly hypoelliptic operator over $\R^n$.
Suppose that $L$ has constant coefficients in $N'\times N$-matrices and is weighted homogeneous.
Then 
\[
\H(\R^n,L) \cap L^p(\R^n,\R^{N}) = \{0\}.
\]
\end{thm}

\begin{proof}
Let $u\in\H(\R^n,L) \cap L^p(\R^n,\R^N)$.
For $\eps\in(0,1]$ define $u_\eps$ as in the proof of Theorem~\ref{thm:liouville}.
We use the same notation as in that proof.
From
\begin{align*}
\|u_\eps\|_{L^p(\R^n)}^p
&=
\int_{\R^n} |u(\eps^{-w_1} x_1,\ldots,\eps^{-w_n} x_n)|^p\,d x_1 \cdots d x_{n} \\
&=
\int_{\R^n} |u( y_1,\ldots, y_n)|^p\,\eps^{w_1 + \cdots + w_n}\,d y_1 \cdots d y_{n} \\
&\le
\|u\|_{L^p(\R^n)}^p < \infty
\end{align*}
we see that $(u_\eps)_\eps$ is an $L^p$-bounded family on $\R^n$.
Using Theorem~\ref{thm:montel} as in the proof of Theorem~\ref{thm:liouville} we find that $u$ is constant.
Since $u$ is also $L^p$ it must be zero.
\end{proof}

\begin{rem}
In the case of scalar constant coefficient hypoelliptic operators Theorems~\ref{thm:liouville} and \ref{thm:Lp-liouville} can also be seen as follows:
If the polynomial $P$ had a zero $x\neq0$, then $P$ would vanish along the curve $t\mapsto (t^{w_1}x_1,\ldots,t^{w_n}x_n)$ by homogeneity.
This would violate H\"ormander's hypoellipticity criterion \cite[Thm.~3.3.I]{H01} for $L=P(\frac{\partial}{\partial x})$.
Thus $x=0$ is the only zero of $P$.
Now \cite[Thm.~2.28]{ES} (whose proof is a simple application of the Fourier transform) says that any solution $u\in\H(\R^n,L)$ must be a polynomial.
If it is in $L^\infty(\R^n)$ or in $L^p(\R^n)$ it must be constant or vanish, respectively.
\end{rem}

\section{Removable Singularities}

Let $M$ be a Riemannian manifold and denote the Riemannian distance of $x,y\in M$ by $d(x,y)$.
For a subset $S\subset M$ let $d(x,S) := \inf_{y\in S}d(x,y)$.
For $r>0$ we denote by
\[
N(S,r) := \{x\in M\mid d(x,S)\le r\} \setminus S
\]
the closed $r$-neighborhood of $S$ with $S$ removed.

\begin{thm}\label{thm:riemann}
Let $S\subset M$ be an embedded submanifold of codimension $m\ge 1$.
Let $L$ be a weakly hypoelliptic operator of order $k\ge 1$ over $M$.
Let $u\in\H(M\setminus S,L)$.
Suppose that for each compact subset $K\subset M$ there exists a constant $C>0$ such that for all sufficiently small $r>0$
\[
\|u\|_{L^1(N(S,r)\cap K)} = \oo(r^{k}) \mbox{ as } r\searrow 0 .
\]
Then $u$ extends uniquely to some $\bar u\in\H(M,L)$.
\end{thm}

\begin{proof}
Uniqueness of the extension is clear because $M\setminus S$ is dense in $M$.
To show existence let $\chi:\R\to\R$ be a smooth function such that 
\begin{itemize}
\item $\chi\equiv 0$ on $(-\infty,1/2]$;
\item $\chi\equiv 1$ on $[1, \infty)$;
\item $0\leq \chi\leq 1$ everywhere.
\end{itemize}
For $r>0$ We define $\chi_r\in C^0(M)$ by
\[
\chi_r(x) := \chi(d(x,S)/r).
\]
Given a compact subset $K\subset M$ the function $\chi_r$ is smooth in a neighborhood of $K$ provided $r$ is small enough.
This is true because the function $x\mapsto d(x,S)$ is smooth on an open neighborhood of $S$ with $S$ removed.

We extend $u$ to a distribution $\bar u\in \D'(M,E)$:
Let $\phi\in \D(M,E^*)$ be a test section.
The compact support of $\phi$ is denoted by $K$.
For $r>0$ sufficiently small $\chi_r\phi\in\D(M\setminus S,E^*)$.
We put
\[
\bar u[\phi] := \lim_{r\searrow 0} u[\chi_r\phi].
\]
The limit exists because for $0<r_1 \le r_2$
\begin{align*}
|u[\chi_{r_1}\phi] - u[\chi_{r_2}\phi]| 
&\leq
\int_{N(S,r_2)\cap K} |u(x)|\cdot |\chi_{r_1}(x)-\chi_{r_2}(x)|\cdot |\phi(x)|\,\vol(x) \\
&\leq
2\cdot\|\phi\|_{C^0(K)}\cdot\|u\|_{L^1(N(S,r_2)\cap K)}  \\
&=
\oo(r_2^k) 
\mbox{ as } r_2 \searrow 0.
\end{align*}
We check that $\bar u$ is a distribution.
Fix a compact subset $K\subset M$.
Then, choosing $r_0>0$ sufficiently small, we obtain
\begin{align*}
\|u\|_{L^1(K\setminus S)}
&=
\|u\|_{L^1(N(S,r_0)\cap K)} + \|u\|_{L^1(\overline{K\setminus N(S,r_0)})} 
< 
\infty .
\end{align*}
Here the first summand is finite because of the assumption in the theorem and the second because $\overline{K\setminus N(S,r_0)}$ is a compact subset of $M\setminus S$.
Hence we find for all $\phi\in\D(M,E^*)$ with $\supp(\phi)\subset K$:
\[
|\bar u[\phi]|
=
\lim_{r\searrow 0} \left| \int_{K\setminus S} \<u(x),\chi_r(x)\phi(x)\> \vol(x) \right|
\le
 \|u\|_{L^1(K\setminus S)} \cdot \|\phi\|_{C^0(K)},
\]
so $\bar u$ is continuous in $\phi$.

It remains to show that $\bar u$ solves $L\bar u=0$ in the distributional sense.
For $\phi\in\D(M,E^*)$ we compute
\begin{align*}
\bar u[L^*\phi]
&=
\lim_{r\searrow 0} \int_{M\setminus S} \<u(x),\chi_r(x)(L^*\phi)(x)\> \vol(x) \\
&=
\lim_{r\searrow 0} \int_{M\setminus S} \<L(\chi_ru)(x),\phi(x)\> \vol(x) \\
&=
\lim_{r\searrow 0} \int_{M\setminus S} \Big\langle\chi_rLu(x)+ \sum_{j=0}^{k-1}P_j(\chi_r)u(x),\phi(x)\Big\rangle \vol(x) \\
&=
\lim_{r\searrow 0}\sum_{j=0}^{k-1} \int_{M\setminus S} \<P_j(\chi_r)u(x),\phi(x)\> \vol(x) \\
&=
\lim_{r\searrow 0}\sum_{j=0}^{k-1} \int_{M\setminus S} \<u(x),P_j(\chi_r)^*\phi(x)\> \vol(x)
\end{align*}
where $P_j(\chi_r)$ is a linear differential operator of order $j$ for each fixed $r$.
It is obtained from the general Leibniz rule.
The coefficients of $P_j(\chi_r)$ depend linearly on $\chi_r$ and its derivatives up to order $k-j$.
Since $\chi_r$ is constant outside $N(S,r)\setminus N(S,r/2)$ the coefficients of $P_j(\chi_r)$ are supported in $N(S,r)\setminus N(S,r/2)$.
For this reason the integration by parts above is justified; there are no boundary terms.
We find
\begin{align*}
\left| \int_{M\setminus S} \<u(x),P_j(\chi_r)^*\phi(x)\> \vol(x) \right|
&\le
C \cdot \|\chi_r\|_{C^{k-j}(K)} \cdot \|\phi\|_{C^j(K)} \cdot \|u\|_{L^1(N(S,r)\cap K)} \\
&\le
C' \cdot r^{j-k} \cdot \|\phi\|_{C^j(K)} \cdot \oo(r^k) \\
&=
\oo(r^j) 
\mbox{ as } r\searrow 0.
\end{align*}
Hence $\bar u[L^*\phi]=0$, i.e., $L\bar u=0$ in the distributional sense.
By weak hypoellipticity of $L$ the extension $\bar u$ must be smooth and solves $L\bar u=0$ in the classical sense.
\end{proof}

\begin{cor}\label{cor:riemann}
Let $L$ be a weakly hypoelliptic operator of order $k\ge 1$ over $M$.
Let $S\subset M$ be an embedded submanifold of codimension $m\ge k+1$.
Let $u\in\H(M\setminus S,L)$ be locally bounded near $S$.

Then $u$ extends uniquely to some $\bar u\in\H(M,L)$.
\end{cor}

\begin{proof}
Since $u$ is locally bounded near $S$ we have for any compact subset $K\subset M$
\begin{align*}
\|u\|_{L^1(N(S,r)\cap K} 
&\le 
\|u\|_{L^\infty(K\setminus S)} \cdot \vol(N(S,r)\cap K)  
\le 
C\cdot \|u\|_{L^\infty(K\setminus S)} \cdot r^m 
= \OO(r^m)
\end{align*}
as $r\searrow 0$.
Since $m\ge k+1$ we get $\|u\|_{L^1(N(S,r)\cap K}=\OO(r^{k+1})$ and therefore $\|u\|_{L^1(N(S,r)\cap K}=\oo(r^{k})$ as $r\searrow 0$.
Theorem~\ref{thm:riemann} yields the claim.
\end{proof}

\begin{ex}
Let $M\subset \C^n$ be an open subset and $S\subset M$ a complex submanifold of complex codimension $\ge 1$.
Then any holomorphic function $u$ on $M\setminus S$, locally bounded near $S$, extends uniquely to a holomorphic function on $M$.
This is Corollary~\ref{cor:riemann} with $k=1$ and $m=2$.
It is classically known as \emph{Riemann's first removable singularity theorem} \cite[Thm.~4.2.1]{S}.

Note that $S$ being a complex submanifold is actually irrelevant; any real submanifold of real codimension $2$ will do.
Moreover, by Theorem~\ref{thm:riemann} one can relax the condition that $u$ be locally bounded near $S$.
A local estimate of the form $|u(x)| \le C\cdot d(x,S)^{-\alpha}$ with $\alpha<1$ is sufficient.
This criterion is sharp because for $M=\C$, $S=\{0\}$, $L=\partial/\partial \bar z$ and $u(z)=1/z$ we have a solution of $Lu=0$ on $M\setminus S$ which satisfies $|u(z)| = d(x,S)^{-1}$ but does not extend across $S$.
\end{ex}

\end{document}